\documentclass[11pt]{article}
\usepackage{a4wide}
\usepackage{amsmath,amssymb,amsfonts,amsthm}
\usepackage[usenames]{color}
\usepackage{graphicx}

\usepackage[colorlinks=true,citecolor=black,linkcolor=black,urlcolor=blue]{hyperref}

\usepackage{float}

\usepackage{enumerate} 

\usepackage{etoolbox}
\makeatletter
\AfterEndEnvironment{thm}{\everypar{\setbox\z@\lastbox\everypar{}}}
\AfterEndEnvironment{theorem}{\everypar{\setbox\z@\lastbox\everypar{}}}
\AfterEndEnvironment{lemma}{\everypar{\setbox\z@\lastbox\everypar{}}}
\AfterEndEnvironment{definition}{\everypar{\setbox\z@\lastbox\everypar{}}}
\AfterEndEnvironment{remark}{\everypar{\setbox\z@\lastbox\everypar{}}}
\AfterEndEnvironment{proposition}{\everypar{\setbox\z@\lastbox\everypar{}}}
\AfterEndEnvironment{corollary}{\everypar{\setbox\z@\lastbox\everypar{}}}
\AfterEndEnvironment{fact}{\everypar{\setbox\z@\lastbox\everypar{}}}
\makeatother

\newcommand{\Gnh}{{G_{n, 1/2}}}

\newtheorem{firsttheorem}{Proposition}

\newtheorem{theorem}[firsttheorem]{Theorem}

\newtheorem{conjecture}[firsttheorem]{Conjecture}

\newtheorem{proposition}[firsttheorem]{Proposition}

\newtheorem*{lemma*}{Lemma}

\usepackage{marginnote}

\usepackage{bbm}
\usepackage[numbers]{natbib}

\newcommand{\mc}{\mathcal}

\newcommand{\Pb}{{\mathbb{P}}}
\newcommand{\Prob}{{\mathbf{P}}}

\renewcommand{\le}{\leqslant}
\renewcommand{\ge}{\geqslant}

\renewcommand{\epsilon}{\varepsilon}

\usepackage{xcolor}

\begin{document}

\title{On a question of Erd\H{o}s and Gimbel on the cochromatic number}

\author{
Annika Heckel\thanks{Matematiska institutionen, Uppsala universitet, Box 480, 751 06 Uppsala, Sweden. Email: \texttt{annika.heckel@math.uu.se}. Funded by the Swedish Research Council, Starting Grant 2022-02829.}
}

\date{December 3, 2024}
\maketitle
	
\begin{abstract}
In this note, we show that the difference between the chromatic and the cochromatic number of the random graph $G_{n,1/2}$ is not whp bounded by $n^{1/2-o(1)}$, addressing a question of Erd\H{o}s and Gimbel.
\end{abstract}

\section{Introduction}
The cochromatic number $\zeta(G)$ of a graph $G$ is the minimum number of colours needed for a vertex colouring where every colour class is either an independent set or a clique. If $\chi(G)$ denotes the usual chromatic number, then clearly $\zeta(G) \le \chi(G)$. Using classical methods, it is not hard to show that for the random graph $G_{n, 1/2}$, with high probability\footnote{As usual, we say that a sequence $(E_n)_{n \ge 0}$ of events holds \emph{with high probability (whp)} if $\Pb(E_n) \rightarrow 1$ as $n \rightarrow \infty$.} we have $\zeta(\Gnh) \sim \chi(\Gnh) \sim \frac{n}{2\log_2 n}$.

 Erd\H{o}s and Gimbel \cite{erdos1993some} (see also \cite{gimbel2016some})  asked the following question: For $G \sim G_{n, 1/2}$, does the difference $\chi(G)-\zeta(G)$ tend to infinity as $n \rightarrow \infty$? In other words, is there a function $f(n) \rightarrow \infty$ such that, with high probability,
\[
\chi(G) - \zeta(G) > f(n)?
\]
At a conference on random graphs in Pozna\'n\footnote{most likely in 1991, or possibly in 1989, based on Erd\H{o}s's and Gimbel's participation records}, Erd\H{o}s offered \$100 for the solution if the answer was `yes', and \$1000 if the answer was `no' (although later said to Gimbel that perhaps \$1000 was too much) \cite{gimbel2016some}. The question is listed as Problem \#625 on Thomas Bloom's Erd\H{o}s Problems website \cite{erdosproblem625}.

In this note, we show that it is \emph{not} the case that the difference $\chi(G) - \zeta(G)$ is whp bounded, and that in fact, it is not whp bounded by $n^{1/2-o(1)}$. It turns out that any function $g(n)$ so that whp $\chi(G)-\zeta(G) \le g(n)$ cannot be smaller than the concentration interval length of the chromatic number $\chi(\Gnh)$, for which corresponding lower bounds were recently obtained \cite{heckel2019nonconcentration,HRHowdoes,heckel2023colouring}. Formally we prove the following statement.

\begin{theorem}\label{theorem:cochromatic}
Let $G \sim G_{n, 1/2}$. There is a constant $c>0$ so that for any sequence of integers $g(n)$ such that 
\begin{equation}\label{eq:gcondition}
\Prob \big(\chi(G) - \zeta(G)  \le g(n) \big) >0.999,
\end{equation}
there is a sequence of integers $n^*$ such that
\[
g(n^*) >  c \frac{\sqrt{n^*} \log \log n^*}{\log^3 n^*}.
\]
\end{theorem}
The proof relies on a comparison of the chromatic and co-chromatic numbers of $G \sim G_{n, 1/2}$ and of the complement graph $\bar{G} \sim G_{n, 1/2}$, which is obtained from $G$ by exchanging all the edges and non-edges. Clearly $\zeta(G)=\zeta(\bar G)$, so for a function $g(n)$ as above, $\chi(G)$ and $\chi(\bar G)$ are likely to be at most $g(n)$ apart from each other. Of course $\chi(G)$ and $\chi(\bar G)$ have the same distribution $\chi(\Gnh)$. Informally speaking, to see why $g(n)$ has to be at least the concentration interval length of this distribution, suppose that we have some lower bound on the concentration interval length; for example, a statement saying that any interval containing $\chi(\Gnh)$ with probability at least 0.9 has length at least $\ell(n)$. Then if $X_1, X_2 \sim \chi(G_{n,1/2})$ were independent samples of this distribution, they would be reasonably likely to be at least about $\ell(n)$ apart from each other. Of course $\chi(G)$ and $\chi(\bar G)$ are not independent, but the first is an increasing  and the other a decreasing function of the edges of $G \sim \Gnh$, and so with the help of Harris's Lemma we can draw the same conclusion. So we know that $\chi(G)$ and $\chi(\bar G)$ are both likely to be at most $g(n)$ apart, but also reasonably likely to be at least $\ell(n)$ apart and it follows that $g(n) \ge \ell(n)$.

Independently from this note, Raphael Steiner recently also discovered the connection between $\chi(G)-\zeta(G)$ and the concentration interval length of $\chi(G_{n, 1/2})$. For his work on this and two other questions of Erd\H{os}, Gimbel and Straight, see \cite{steiner2024cochromatic}.

\section{Proof of Theorem~\ref{theorem:cochromatic}}

Turning to the details, let us first state the non-concentration result for $\chi(\Gnh)$ that we will use, which follows by combining Theorem 8 from \cite{HRHowdoes} and Theorem 1.2 from \cite{heckel2023colouring}.

\begin{theorem}[\cite{HRHowdoes,heckel2023colouring}] \label{theorem:nonconcentration}
	There is a constant $c>0$ so that for any sequence of intervals $[s_n, t_n]$ such that $\Prob \big(\chi(G_{n, 1/2}) \in [s_n, t_n])>0.9$, there is a sequence of integers $n^*$ such that
	\[
	t_{n^*} - s_{n^*} > c \frac{\sqrt{n^*} \log \log n^*}{\log^3 n^*}.
	\]
\end{theorem}

Theorem~\ref{theorem:cochromatic} follows directly from Theorem~\ref{theorem:nonconcentration} and the following proposition.

\begin{proposition}\label{prop:hilfs}
	Let $g(n)$ be a sequence of integers which satisfy \eqref{eq:gcondition}, then there is a sequence of intervals $[s_n, t_n]$ with $t_n-s_n=g(n)$ so that 
	\[\Prob \big(\chi(G_{n, 1/2}) \in [s_n, t_n])>0.9.\]
\end{proposition}

\begin{proof}[Proof of Proposition~\ref{prop:hilfs}]
Let $G\sim \Gnh$ and let $\bar G$ be the complement graph of $G$ (which contains exactly the edges which are missing in $G$). Then $\bar G \sim \Gnh$ and $\zeta(\bar G)=\zeta(G)$, so  with probability at least 0.999, 
\begin{equation}\label{eq:close}
\chi(\bar G) \le \zeta(\bar G)+g(n) = \zeta(G)+g(n) \le \chi(G)+g(n).
\end{equation}

Now let $s_n$ be the smallest  integer $k$ such that $\Prob(\chi(G) \le k) \ge 0.05$, and define the following events:
\begin{align*}
	\mc D &= \{\chi(G) \le s_n\} ,\\
	\mc U &= \{\chi(\bar G) \le s_n+g(n)\}.
\end{align*}
Then $\mc D$ is a down-set and $\mc U$ is an up-set in the edges of $G \sim \Gnh$. By the definition of $s_n$, \[\Prob(\mc D) \ge 0.05.\] Furthermore, if $\mc D$ holds, then either \eqref{eq:close} does not hold (which has probability at most 0.001), or \eqref{eq:close} holds and implies $\mc U$, so
\[\Prob(\mc U \cap \mc D ) \ge \Prob(\mc D) - \Prob\big(\chi(\bar G) > \chi(G)+g(n)\big) \ge \Prob(\mc D)-0.001.\]
By Harris's Lemma (see for example \S2, Lemma~3 in \cite{bollobas2006percolation}),
\[
\Prob(\mc U) \ge \Prob(\mc U \cap \mc D) /\Prob(\mc D) \ge 1-\frac{0.001}{\Prob(\mc D)} \ge 1-\frac{0.001}{0.05}= 0.98.
\]
But since $G$ and $\bar G$ have the same distribution, and by the definition of $s_n$, this implies
\begin{align*}
\Prob \big(s_n \le \chi(G) \le s_n+g(n) \big) &= 1 - \Prob\big(\chi(G) \le s_n-1\big) -\Prob\big(\chi(G) >s_n+g(n)\big) \\
&\ge 1-0.05-0.02 > 0.9.
\end{align*}
The claim follows.
\end{proof}
\section{Discussion}

So how about Erd\H{os} and Gimbel's original question: does $\chi(G)-\zeta(G)$ tend to infinity whp for $G \sim \Gnh$? Theorem~\ref{theorem:cochromatic} suggests that the answer is `yes', but of course does not imply~this.  

If we had a result like Theorem~\ref{theorem:cochromatic}, but with the conclusion that $g(n) \ge h(n)$ for every $n$ and some $h(n) \gg \sqrt{n}/ \log n$, this would imply that the answer to the original question is `yes': by an argument of Alon \cite{alonspencer,scott2008concentration},  both $\chi(\Gnh)$ and $\zeta(\Gnh)$ are contained in a sequence of intervals of length about $\sqrt{n}/\log n$, respectively, and consequently so is their difference.\footnote{Formally the statement is: For $G \sim G_{n, 1/2}$ and any function $\omega(n) \rightarrow \infty$, there is a sequence of intervals of length at most $\omega(n) \sqrt{n} / \log n$ that contains $\chi(G)-\zeta(G)$ whp.} Let $[s_n, t_n]$ be such a sequence of intervals, then taking $g(n)=t_n$ would give that $t_n \ge h(n) \gg  \sqrt{n}/ \log n$, which implies that $s_n \gg \sqrt{n}/ \log n$ as well\footnote{for a suitable choice of the arbitrary function $\omega(n) \rightarrow \infty$}, and so whp $\chi(G) -\zeta(G) \ge s_n \gg  \sqrt{n}/ \log n$.

It is reasonable to expect that the chromatic number $\chi(\Gnh)$ is close to its \emph{first moment threshold}, that is, the smallest $k$ such that the expected number of $k$-colourings is at least $1$.\footnote{This is an oversimplification and there are several complications with this heuristic, for a detailed discussion see \S1.3 and the appendix in \cite{HRHowdoes}.} Using the same heuristic for the cochromatic number $\zeta(\Gnh)$, the first moment threshold there should be of order $n/\log^3 n$ smaller than that of the chromatic number: for any $k \sim n/(2\log_2 n)$, the expected number of $k$-cocolourings is multiplied by a factor $2^k= \exp(\Theta(n/\log n))$ when compared to the expected number of $k$-colourings (since we may choose for each colour class whether it is a clique or an independent set); and decreasing the number of colours by 1 should multiply the expectation by a factor $\exp(-\Theta (\log^2 n))$ (like it does for the chromatic number\footnote{Again an oversimplification, and we refer to the discussion in~\cite{HRHowdoes}.}). We therefore make the following conjecture.

\begin{conjecture}
For $G \sim \Gnh$, whp, 
\[\chi(G)-\zeta(G)= \Theta(n/\log^3 n).\]
\end{conjecture}
\section*{Acknowledgements} I would like to thank John Gimbel for bringing this problem to my attention and for helpful discussions, and Svante Janson for his help in tracking down past conference proceedings.

\end{document}